\documentclass{amsart}
\usepackage{amsmath,amssymb,hyperref,mathrsfs,color,calc}
\usepackage{paralist}
\usepackage[all,cmtip]{xy}
\usepackage{graphicx,float,wrapfig,tikz}
\usepackage[font=small]{caption, subcaption}
\usetikzlibrary{intersections}
\usepackage{cite}
\usepackage[misc]{ifsym}
 \newtheorem{The}{Theorem}[section]
 \newtheorem{Cor}[The]{Corollary}
 \newtheorem{Lem}[The]{Lemma}
 \newtheorem{Pro}[The]{Proposition}
 \theoremstyle{definition}
 \newtheorem{Def}[The]{Definition}
 \theoremstyle{remark}
 \newtheorem{Rem}[The]{Remark}
 
 \numberwithin{equation}{section}
\newcommand{\T}{\mathbb{T}}
\newcommand{\R}{\mathbb{R}}
\newcommand{\Z}{\mathbb{Z}}
\newcommand{\N}{\mathbb{N}}

\begin{document}
\title[Closed geodesics pass through poles]{On class A Lorentzian 2-tori with poles I: Closed geodesics pass through poles}
\author{Lu Peng, Liang Jin \and Xiaojun Cui}
\thanks{X. Cui is supported by the National Natural Science Foundation of China (Grants 11571166, 11631006, 11790272), the Project Funded by the Priority Academic Program Development of Jiangsu Higher Education Institutions (PAPD) and the Fundamental Research Funds for the Central Universities.\\
L. Peng, X. Cui (\Letter): Department of Mathematics, Nanjing University, Nanjing 210093, China. e-mail: xcui@nju.edu.cn\\
L. Peng: e-mail: penglu1991@gmail.com\\
L. Jin: Academy of Mathematics and Systems Science, CAS, Beijing 100190, China. e-mail: jinliang@amss.ac.cn}

\date{\today}

\begin{abstract}
  In this paper, by studying certain isometries on globally hyperbolic planes, we prove that if $p$ is a timelike pole on a class A Lorentzian 2-torus, then there exists a closed timelike geodesic passing through $p$ with any preassigned free homotopy class in the interior of the stable time cone. We also show a non-rigid result when timelike poles appear.
\end{abstract}
\maketitle

\section{Introduction and statement of the main results}
In Riemannian geometry, a point $p$ on a complete Riemannian manifold $(M,g_R)$ is a pole if no geodesic $\gamma:[0,\infty)\rightarrow M$ with $p=\gamma(0)$ contains a pair of conjugate points, or equivalently, $\text{exp}_{p}:T_{p}M\rightarrow M$ is a covering map. Thus the universal cover of $M$ is diffeomorphic to $\mathbb{R}^{n}$, where $n=\text{dim} M$, see also \cite{Carmo}. Among such manifolds, there are certain ones admit no metric with negative sectional curvature, the most typical examples are tori. One may expect the presence of poles on Riemannian tori is, in some sense, special and would have rigid effects on the metric structure. From this point of view, many studies show that this is indeed the case. Perhaps the most celebrated result is due to E. Hopf, which says that any Riemannian $n$-torus with no conjugate points (i.e., all points are poles) is flat. In the paper \cite{Hopf}, Hopf himself proved the case when $n=2$. As a landmark success of global differential geometry and dynamical system, the proof for the higher dimensional case was finally completed by D. Burago and S. Ivanov in \cite{Burago-Ivanov}.

On the way of the pursuit of Hopf's theorem, much efforts have been paid on the investigation of dynamical behaviors of geodesics when poles appear. In his paper \cite{Innami}, N. Innami proved that for any pole on a Riemannian $n$-torus, there exists a closed geodesic passing through it with any preassigned free homotopy class. Recently in \cite{Bangert10}, V. Bangert showed the similarities between the dynamics of the geodesic flow of a 2-torus with a pole and the integrable dynamics of a torus of revolution. It is the motivation of this paper and our next publication ``On class A Lorentzian 2-tori with poles II: Foliations of tori by timelike lines" to reveal similar results of \cite{Innami} and \cite{Bangert10} under the setting of Lorentzian geometry.

It is plausible that the notion of pole has the following counterpart in Lorentzian geometry,
\begin{Def}
Let $(M,g)$ be a Lorentzian manifold, a point $p$ on $(M,g)$ is a timelike pole if no timelike geodesic $\gamma:[0,a)\rightarrow M$ starting from $p=\gamma(0)$ contains a pair of conjugate points.
\end{Def}

\begin{Rem}
One could also define the notion of spacelike pole, but this is not the focus of this paper.
\end{Rem}

Comparing with Riemannian cases, the dynamics of geodesics on Lorentzian tori are more complicated. Indeed, we have to rebuild some properties that are relatively apparent in the positive-definite cases. Also, because of these reasons, we choose not to focus on all but some typical kind of two dimensional Lorentzian tori, namely

\begin{Def}\label{class A}
A Lorentzian 2-torus $(\mathbb{T}^{2},g)$ is called class A if it is totally vicious (i.e., every point lies on a timelike loop) and its Abelian cover $(\mathbb{R}^{2},g)$ is globally hyperbolic.
\end{Def}

\begin{Rem}
Class A 2-tori as well as their higher dimensional generalizations, class A spacetimes are proved to be suitable choices for developing global variational methods in the setting of Lorentzian geometry, see \cite{Sc,Su1,Su2,Su3} for more details.
\end{Rem}

Throughout the paper, we shall always denote $\pi:(\mathbb{R}^{2},g)\rightarrow(\mathbb{T}^{2},g)$ the Abelian cover of a class A Lorentzian 2-torus $\mathbb{T}^{2}$. The deck transformations $(\cong \mathbb{Z}^2)$ associated to $\pi$ act on $\mathbb{R}^2$ and are defined by
\begin{displaymath}
T: (k,x)\in\mathbb{Z}^2\times\mathbb{R}^2\mapsto x+k \in\mathbb{R}^2.
\end{displaymath}
$T_k:=T|_{k\times\mathbb{R}^2}$ is the translation by $k\in\mathbb{Z}^2$ and is an isometry for every $k\in\mathbb{Z}^2$ with respect to $g$. Denote by $\Gamma:=\{T_k:k\in\mathbb{Z}^2\}$ the group of deck transformations associated to $\pi$ on $\mathbb{R}^2$. We call a timelike geodesic $\gamma$ on $(\R^2,g)$ periodic if $T_k\gamma=\gamma$ for some deck transformation $T_k\in\Gamma,k\neq 0$. Using the notions of rational asymptotic direction and maximal geodesic in Section 2, we prove that

\begin{The} \label{Thm3}
Let $(\T^2,g)$ be a class A Lorentzian 2-torus with a timelike pole $p$, then there exists a maximal closed timelike geodesic passing through $p$ with any preassigned free homotopy class in the interior of the stable time cone. More precisely, if $(\R^2,g)$ is the Abelian cover of $(\T^2,g)$ and $\bar{p}\in(\R^2,g)$ is any lift of $p$, then for any rational asymptotic direction $\alpha\in(m^-,m^+)$, there exists a maximal periodic timelike geodesic passing through $\bar{p}$ with asymptotic direction $\alpha$.
\end{The}

Like Riemannian cases, the appearance of timelike poles does not imply the metric rigidity, as the following theorem shows

\begin{The}\label{Thm1}
For any $\varepsilon$, $0<\varepsilon<1$, there is a non-flat class A 2-torus $(\mathbb{T}^{2},g)$ such that $(1-\varepsilon)\text{Vol }(\mathbb{T}^{2},g)\leq\lambda\leq\text{Vol }(\mathbb{T}^{2},g)$, where $\lambda$ is the volume of the set of all timelike poles on $(\mathbb{T}^{2},g)$ and $\text{Vol }(\mathbb{T}^{2},g)$ is the total volume of $(\mathbb{T}^{2},g)$.
\end{The}

\begin{Rem}
In contrast with Riemannian tori, the Hopf's theorem does not even hold for Lorentzian 2-tori! The so-called Clifton-Pohl torus offers such an example. However, it is the only non-flat Lorentzian 2-torus without conjugate points as far as we know and is not of class A.

On the other hand, in \cite{Bavard-Mounoud} and \cite{Su2}, the authors had shown that all Lorentzian class A metrics on a 2-torus are contained in one connected component of the space of Lorentzian metrics on the 2-torus and there exist metrics without conjugate points in each such component. But the non-flat metrics without conjugate points constructed in \cite{Bavard-Mounoud} are not geodesically complete. Hence, to the best of our knowledge, the existence of geodesically complete non-flat Lorentzian tori without conjugate points is still unknown and we have reason to believe that the Hopf's theorem may be true in the category of class A Loretzian 2-tori, especially with the geodesic completeness condition.
\end{Rem}

The paper is organized as follows. In Section $2$, we equip ourself with enough prerequisites on class A 2-tori for further discussions. In Section 3, we construct a non-flat class A Lorentzian metric on $\T^2$ with timelike poles, then give a proof of Theorem \ref{Thm1}. Section 4 is devoted to a study of certain isometries called $P$-motions on globally hyperbolic spacetimes. In Section 5, we further discuss $P$-motions on the Abelian cover of class A 2-tori. Using Lorentzian Busemann functions, we prove that two axes whose asymptotic directions in $(m^-,m^+)$ of the same $P$-motion are parallel. In Section $6$, combining our results in Sections $4$ and $5$, we complete the proof of Theorem \ref{Thm3}.

\section{Preliminaries}
In this section, we introduce elements in Lorentzian geometry as well as previous results on class A 2-tori that are necessary for our presentation. We recommend the textbook \cite{B-E-E} and papers \cite{Cui-Jin,Jin-Cui1,Jin-Cui2,Sc,Su1,Su2,Su3} for a comprehensive reference on these topics.

\subsection{A brief introduction to global Lorentzian geometry}
~\\
\textit{2.1.1. Spacetimes and causal relations:}

We shall refer standard notions in Lorentzian geometry to \cite{B-E-E}. Let $(M,g)$ be a smooth connected Lorentzian manifold. A nonzero tangent vector $v\in TM$ is called timelike (resp. null, spacelike) or causal if $g(v,v)<0 (\text{resp.}=0, >0) \text{ or}\leq 0$. A smooth curve is said to be timelike (resp. null, spacelike) or causal if its tangent vectors are always timelike (resp. null, spacelike) or causal. $(M,g)$ is time oriented if $M$ admits a smooth timelike vector field $X$, any time oriented Lorentzian manifold is called a spacetime. A causal vector $v\in T_pM$ is called future (resp. past) directed if $g(X(p),v)<0(\text{resp.} >0)$. A smooth curve $\gamma$ of $(M,g)$ is called future (resp. past) directed if its tangent vectors are always future (resp. past) directed. Note that a smooth causal curve in a spacetime $(M,g)$ is either future directed or past directed, see \cite[p.54]{B-E-E}.

Throughout this paper, we always consider spacetimes, rather than general Lorentzian manifolds, and future directed curves if there is no additional explanation.

Let $(M,g)$ be a spacetime and $(H,h)$ be a Riemannian manifold, we recall that the Lorentzian product of $(M,g)$ and $(H,h)$ is the manifold $M\times H$ equipped with the Lorentzian metric $\hat{g}:=\pi^{\ast}g+\eta^{\ast}h$, where $\pi:M\times H\rightarrow M$ and $\eta:M\times H\rightarrow H$ denote the canonical projection maps. It is easy to see that $(M\times H,\hat{g})$ is also a spacetime by \cite[Lemma 3.54]{B-E-E}.

For $p,q\in M$, we say $p\ll q$ if there is a piecewise smooth future directed timelike curve from $p$ to $q$, and $p\leq q$ if either $p=q$ or there is a piecewise smooth future directed causal curve from $p$ to $q$. The chronological future of $p$ is $I^+(p):=\{q\in M: p\ll q\}$, the causal future of $p$ is $J^+(p):=\{q\in M: p\leq q\}$. The chronological and causal past are similarly defined by reversing the time orientation. For a subset $S\subseteq M$, $I^{\pm}(S):=\bigcup_{p\in S}I^{\pm}(p)$ and a curve $\gamma$, $I^\pm(\gamma):=I^\pm(\text{Im}(\gamma))$.

Note that both $I^+(p)$ and $I^-(p)$ are open sets for any point $p\in M$ and the relations $\ll,\leq$ are transitive. Moreover, we have
\begin{equation*}
  p\ll q \quad\text{and}\quad  q\leq r \quad\text{implies}\quad  p\ll r,
\end{equation*}
and
\begin{equation*}
  p\leq q \quad\text{and}\quad q\ll r \quad\text{implies}\quad  p\ll r.
\end{equation*}
~\\
\textit{2.1.2. Lorentzian distance function and its derivatives:}

Let $\gamma:[a,b]\rightarrow M$ be a continuous and piecewise smooth causal curve and $a=a_{0}<...<a_{n}=b$ a partition of $[a,b]$ such that $\gamma|_{[a_{i},a_{i+1}]},i=0,...,n-1$ is smooth, the Lorentzian arclength of $\gamma$ associated to $g$ is
\begin{equation*}
L^{g}(\gamma):=\sum_{i=1}^{n}\int_{a_{i}}^{b_{i}}\sqrt{-g(\dot{\gamma},\dot{\gamma})(t)}dt.
\end{equation*}
Now if $q\in J^+(p)$, the Lorentzian distance function $d:M\times M\rightarrow\mathbb{R}$ associated to $g$ is defined by
\begin{equation*}
d(p,q):=\sup\{L^{g}(\gamma):\gamma\in\mathcal{C}^{+}(p,q)\},
\end{equation*}
where $\mathcal{C}^{+}(p,q)$ denotes the set of all future directed continuous and piecewise smooth causal curves from $p$ to $q$; if $q\notin J^+(p)$, $d(p,q)$ is defined to be zero. From the above definition,
\begin{itemize}
  \item $d(p,q)>0$ if and only if $q\in I^+(p)$,
  \item $d$ satisfies the reverse triangle inequality, i.e., if $p\leq q\leq r$, then
        \begin{equation}\label{rt}
        d(p,r)\geq d(p,q)+d(q,r).
        \end{equation}
\end{itemize}

We could also define distance between a point $p\in M$ and a set $S\subseteq M$ or two sets $S_1,S_2\subseteq M$ by
\begin{equation}\label{diss}
d(p,S):=\sup_{q\in S}\, d(p,y),\hspace{0.6cm}d(S_1,S_2):=\sup_{p\in S_1,q\in S_2} d(p,q).
\end{equation}
If $p\ll f,f\in S$ and $d(p,f)=d(p,S)<\infty$, we call $f$ a foot of $p$ on $S$.
~\\

The regularity of Lorentzian distance function comes from reasonable conditions on the causal structures of spacetimes. Perhaps the most important one is the global hyperbolicity, which plays a similar role in Lorentzian geometry as the completeness plays in Riemannian geometry. Instead of giving the definition, we just state two of its consequences, they are implicitly used here and there in this paper.
\begin{Pro}[{\cite[Lemma 4.5, Theorem 6.1]{B-E-E}}]\label{globally-hyperbolic}
Let $(M,g)$ be a globally hyperbolic spacetime, then
\begin{enumerate}
  \item the Lorentzian distance function $d$ is continuous and finite on $M\times M$,
  \item for any $p,q\in M$ with $p\leq q$, there is a causal geodesic $\gamma$ from $p$ to $q$ with $L^{g}(\gamma)=d(p,q)$. Such geodesics are called maximal geodesics, or simply called maximizers.
\end{enumerate}
\end{Pro}

Let $\gamma:(a,b)\rightarrow M$ be a timelike (causal) curve, the point $p$ is called a future (resp. past) endpoint of $\gamma$ if the limit $\lim\limits_{t\rightarrow b^-}\gamma(t)$ (resp. $\lim\limits_{t\rightarrow a^+}\gamma(t)$) exists and equals to $p$. A timelike (causal) curve is called future (resp. past) inextendible if it has no future (resp. past) endpoint. We call a timelike (causal) curve $\gamma:(a,b)\rightarrow M$ inextendible if it is both future and past inextendible.

With the notion of inextendibility, we have
\begin{Def}\label{Def_line}
A timelike (causal) geodesic $\gamma:[a,b)\rightarrow M$ is a timelike (causal) ray if it is future inextendible and maximal; a timelike (causal) geodesic $\gamma:(a,b)\rightarrow M$ is a timelike (causal) line if it is inextendible and maximal. We call a timelike (causal) geodesic future complete if its affine parameter could be extended to $+\infty$.

Moreover, let $(\mathbb{R}^2,g)$ be the Abelian cover of a class A 2-torus $(\mathbb{T}^2,g)$, a timelike (causal) geodesic on $(\mathbb{T}^2,g)$ is called a timelike (causal) maximizer (ray, line) if any of its lift is a timelike (causal) maximizer (ray, line).
\end{Def}

On a non-compact spacetime, one could define functions which measure the distance to a point at infinity, now known as Lorentzian Busemann functions. For future use, we give the definition and its basic properties here.
\begin{Def}[\cite{Galloway}]
Let $(M,g)$ be a non-compact spacetime and $\gamma:[0,+\infty)\rightarrow M$ be a future complete timelike ray. The Lorentzian Busemann function $b_{\gamma}:M\rightarrow \R\cup\{\pm\infty\}$ associated to $\gamma$ is defined by
\begin{equation}\label{busemann}
b_{\gamma}(p)=\lim\limits_{s\rightarrow \infty}[s-d(p,\gamma(s))].
\end{equation}
\end{Def}

\begin{Pro}[\cite{Galloway}]\label{pro-buse}
Let $(M,g)$ be a non-compact spacetime and $\gamma:[0,+\infty)\rightarrow M$ a future complete timelike ray, then
\begin{enumerate}
  \item the limit in equation \eqref{busemann} always exists by the reverse triangle inequality \eqref{rt};
  \item $b_{\gamma}(p)<+\infty$ if and only if $p\in I^-(\gamma)$, if $p\in I^+(\gamma(0))\cap I^-(\gamma)$, $b_{\gamma}(p)\geq d(\gamma(0),p)>0$;
  \item for any $p\leq q$ in $I^-(\gamma)$, we have
        \begin{equation}\label{bxy}
        b_{\gamma}(q)-b_{\gamma}(p)\geq d(p,q).
        \end{equation}
\end{enumerate}
\end{Pro}

Level sets of $b_{\gamma}$ are called horospheres, they are especially useful when studying geometrical behaviors of timelike rays or lines.
\begin{Def}\label{horo}
Let $(M,g)$ be a non-compact spacetime and $\gamma:[0,+\infty)\rightarrow M$ a future complete timelike ray, for $p\in I^-(\gamma)$, set
\begin{equation*}
K_{\infty}(p,\gamma):=\{x\in I^{-}(\gamma): b_{\gamma}(x)=b_{\gamma}(p)\}
\end{equation*}
and call it the horoshpere through $p$ with central timelike ray $\gamma$.
\end{Def}

\subsection{Previous results on class A 2-tori}
Let $(M^2,g)$ be a connected, closed orientable two dimensional spacetime and $g_{R}$ an auxiliary complete Riemannian metric on $M^{2}$. It's well known that $M^2$ is diffeomorphic to $\mathbb{T}^2$, we denote by $\|\cdot\|$ and dist$_{\|\cdot\|}(\cdot,\cdot)$ the stable norm associated to $g_R$ (see \cite{Burago} for its definitions) and the metric induced by $\|\cdot\|$ on $H_{1}(\mathbb{T}^2,\mathbb{R})$.

Now we recall the definition of asymptotic directions: we think of $H_{1}(\mathbb{T}^2,\mathbb{R})$ as a vector space and set $S^1:=\{h|h\in H_1(\mathbb{T}^2,\mathbb{R}),\|h\|=1)\}$. Denote by $\overline{\alpha}$ the unique half line emanating from the origin and containing the vector $\alpha\in S^1$ in $H_1(\mathbb{T}^2,\mathbb{R})$.

\begin{Def}\label{Def_a.d.}
Let $\gamma:I\rightarrow \mathbb{T}^2$ be a future inextendible causal curve. We say $\gamma$ or its lift $\tilde{\gamma}$ to $\R^{2}$ has an asymptotic direction $\alpha$ if there exist constant $D(\gamma,g,g_R)>0$ and $\alpha\in S^1$ such that
\begin{equation*}
\text{dist}_{\|\cdot\|}(\tilde{\gamma}(t)-\tilde{\gamma}(s),\overline{\alpha})\leq D(\gamma,g,g_R)
\end{equation*}
for all $[s,t]\subseteq I$, where $\tilde{\gamma}(t)-\tilde{\gamma}(s)\in H_1(\mathbb{T}^2,\mathbb{R})$. We call $\alpha$ rational if $\overline{\alpha}\cap H_1(\mathbb{T}^2,\mathbb{Z})_\mathbb{R}$ is non-empty.
\end{Def}

Since $(\T^2,g)$ is a spacetime, the Lorentzian metric $g$ naturally gives rise to two future directed lightlike $C^{\infty}$ vector fields $X^{\pm}$ and we choose $\{X^-(p),X^+(p)\}$ to be positively oriented at each point $p\in\mathbb{T}^{2}$, see \cite{Su2}. The integral curves of $X^{+}$ (resp. $X^{-}$) are lightlike and, by elementary foliation theory on surfaces, they all have a same asymptotic direction $m^{+}$ (resp. $m^{-}$). The first definition of class A 2-tori was given by E. Schelling and recently proved to be equivalent to Definition \ref{class A}:

\begin{Pro}[{\cite{Sc, Su2}}]\label{eq class A}
Let $(\mathbb{T}^2,g)$ be a Lorentzian 2-torus, then it is of class A if and only if $m^{+}\neq m^{-}$.
\end{Pro}

Let $(\mathbb{T}^2,g)$ be a class A 2-torus, we set $\mathfrak{T}\subseteq H_1(\mathbb{T}^2,\mathbb{R})$ to be the convex hull of $\overline{m}^-\cup \overline{m}^+$ and call it the stable time cone of $(\mathbb{T}^2,g)$. Given the usual orientation (i.e., the counter-clockwise orientation is positive) on $H_1(\mathbb{T}^2,\mathbb{R})\cong\mathbb{R}^2$, we define an order on $S^1\cap\mathfrak{T}$ by: for any $\alpha,\beta\in S^1\cap\mathfrak{T}$,  $\alpha<\beta$ if and only if $\{\alpha,\beta\}$ is positively oriented. Then $(S^1\cap \mathfrak{T},<)$ is isomorphic to a closed interval which we denote by $[m^-,m^+]$, where $m^-<m^+$ are endpoints of $S^1\cap\mathfrak{T}$. Define $(m^-,m^+):=[m^-,m^+]\backslash\{m^{\pm}\}$. In sense of the following lemma, $[m^-,m^+]$ is called the set of causal asymptotic directions and we call $(m^-,m^+)$ the set of timelike asymptotic directions.

\begin{Lem}[{\cite{Sc, Su1}}]\label{thm_Schelling}
Let $(\R^2,g)$ be the Abelian cover of a class A Lorentzian 2-torus, then
\begin{enumerate}
  \item any asymptotic direction of a causal curve (if it has) falls in $[m^-,m^+]$,
  \item for every $\alpha\in[m^-,m^+]$ there exist causal lines with asymptotic direction $\alpha$ and every causal line has an asymptotic direction $\alpha\in[m^-,m^+]$.
\end{enumerate}
\end{Lem}

Let $\gamma$ be an inextendible causal curve on $(\mathbb{R}^2,g)$, then $\R^2\backslash Im(\gamma)=U^-\cup U^+$, where $U^\pm$ are two connected components of $\R^2\backslash Im(\gamma)$ that are diffeomorphic to $\R^2$. One of these components, say $U^{+}(U^{-})$ satisfies the following condition: for every spacelike smooth curve $\xi$ that initiates from $p=\gamma(0)$ such that $\xi(t),t>0$ is contained in $U^{+}$ (resp. $U^{-}$), $\{\dot{\gamma}(0),\dot{\xi}(0)\}$ is positively (resp. negatively) oriented.

By the above discussions which strongly rely on the topological structure of the plane, we define a useful relation between an inextendible causal curve and a point outside, see also \cite[Definition 1.11]{Jin-Cui1}.
\begin{Def}\label{def_point_line}
We say a point $p\in\R^2\backslash Im(\gamma)$ satisfies $p>\gamma(p<\gamma)$ if and only if $p\in U^+(p\in U^-)$.
\end{Def}

Now we recall some elementary properties of timelike rays and lines on the Abelian cover of a class A 2-torus, which were discussed in the paper \cite{Jin-Cui1} and will be frequently used in this paper. We list them in
\begin{Lem} \label{lem_I_gamma}
Let $(\R^2,g)$ be the Abelian cover of a class A Lorentzian 2-torus, then
\begin{enumerate}
  \item Two inextendible causal lines in $(\R^2,g)$ with different asymptotic directions intersect each other exactly once.
  \item Every timelike line $($resp. ray$)$ $\gamma:I\rightarrow (\mathbb{R}^2,g)$ with asymptotic direction $\alpha\in(m^-,m^+)$ is complete $($resp. future complete$)$.
  \item Every timelike line $\gamma:\R\rightarrow (\R^2,g)$ with asymptotic direction $\alpha\in(m^-,m^+)$ satisfies
        \begin{eqnarray*}
        I^+(\gamma)&=&\bigcup\limits_{k\in\N}I^+(\gamma(-k))=\R^2, \\
        I^-(\gamma)&=&\bigcup\limits_{k\in\N}I^-(\gamma(k))=\R^2.
        \end{eqnarray*}
\end{enumerate}
\end{Lem}

\section{Proof of Theorem \ref{Thm1}}
This section is devoted to a proof of Theorem \ref{Thm1}. First, we observe that for every class A 2-torus $(\mathbb{T}^{2},g)$, $(\mathbb{T}^{2},-g)$ is also class A but with interchanged timelike and spacelike vectors. It is immediately from the observation that the timelike poles of $(\mathbb{T}^{2},-g)$ are the spacelike poles of $(\mathbb{T}^{2},g)$ and Theorem \ref{Thm1} remains true if we replace the notion of timelike pole by that of spacelike pole.

To begin our construction, let $(t,x)$ denote the usual linear coordinates on $\R^2$ and $(\dot{t},\dot{x})$ the coordinates in each tangent space with respect to the natural frame $\{\partial/\partial t,\partial/\partial x\}$. Fix a smooth function $f:\R\rightarrow (0,\infty)$, we define a Lorentzian metric on the plane by
\begin{equation}
g_f=-f^{2}(x)dt^{2}+dx^2.
\end{equation}
By definition, $(\R^2,g_f)$ is a spacetime oriented by $\partial/\partial t$.

Consider the unit future observer bundle $T_{-1}\R^2$ associated to $g_f$ (see Section 6, Definition \ref{Def_cutpoint} or \cite[Section 9.1]{B-E-E}), i.e., all future directed tangent vectors such that $f^{2}(x)\dot{t}^{2}-\dot{x}^2=1$. This means that each fiber of $T_{-1}\R^2$ is a branch of a hyperbola, so we could choose the parametrization $\dot{t}=\frac{\cosh\psi}{f(x)}>0,\dot{x}=\sinh\psi$ of the fiber of $T_{-1}\R^2$ at $(t,x)$ in terms of the so called hyperbolic angle $\psi\in\mathbb{R}$. Thus
\begin{equation*}
T_{-1}\R^2=\{(t,x;\frac{\cosh\psi}{f(x)},\sinh\psi)|(t,x)\in\R^{2},\psi\in\mathbb{R}\}.
\end{equation*}

Let $-\infty\leq a<b\leq+\infty$ and $\gamma:(a,b)\rightarrow(\R^2,g_f);\gamma(\tau):=(t_{\gamma}(\tau),x_{\gamma}(\tau))$ be a unit speed timelike inextendible geodesic on $(\R^2,g_f)$. Denote by $\dot{\gamma}$ the derivative of $\zeta$ with respect to $\tau$. Since $\gamma$ is future directed and unit speed, for any $\tau\in(a,b)$, $\dot{t}_{\gamma}(\tau)>0$ and $\dot{\gamma}(\tau)=(\dot{t}_{\gamma}(\tau),\dot{x}_{\gamma}(\tau))\in T_{-1}\R^2$, thus there is hyperbolic angle $\psi_\gamma(\tau)$ such that $\dot{\gamma}(\tau)=(\frac{\cosh\psi_\gamma(\tau)}{f(x_{\gamma}(\tau))},\sinh\psi_\gamma(\tau))$ or
\begin{equation}\label{E1}
\cosh\psi_\gamma(\tau)=f(x_{\gamma}(\tau))\dot{t}_{\gamma}(\tau).
\end{equation}
On the other hand, $\gamma$ satisfies the geodesic equation
\begin{equation*}
\frac{d^{2}t}{d\tau^{2}}+\frac{2f'(x)}{f(x)}\frac{dx}{d\tau}\frac{dt}{d\tau}=0,
\end{equation*}
since $f>0$, it follows that $f^{2}(x)\cdot\frac{dt}{d\tau}\equiv const$ along $\gamma$. Combining with $(\ref{E1})$, we obtain that for any $\tau\in(a,b)$,
\begin{equation}\label{E2}
f(x_{\gamma}(\tau))\cosh\psi_\gamma(\tau)\equiv const>0.
\end{equation}

We choose the positive orientation on $\R^2$ as before and state the following
\begin{Lem}\label{Lemma1}
Let $f$ be a smooth function with positive lower and upper bound and $(\R^2,g_f)$ be defined as above. If $p=(t_0, x_0)\in\R^2$ and $x_0$ is a maximal point of f, i.e., $f(x_0)=\sup\limits_{x\in\R} f(x)$, then $p$ is a timelike pole on $(\R^2,g_f)$.
\end{Lem}

\begin{proof}
From (\ref{E1}) and (\ref{E2}), we know that for any unit speed timelike geodesic, both $\frac{dt}{d\tau}$ and $\frac{dx}{d\tau}$ are bounded, thus $(\mathbb{R}^{2},g_f)$ is timelike geodesically complete. All we need to prove is that any two distinct timelike geodesics emanating from $p$ will not intersect each other at any point except $p$. Since translations and reflections about $t$ are $g_f$-isometries, we only consider future parts of the geodesics emanating from $p$.

It follows from $x_0$ is a maximal point of $f$ that the curve $\gamma_{0}(\tau):=(t_{0}+\frac{\tau}{f(x_{0})},x_{0}),\tau\in\R$ is a unit speed timelike geodesic through $p$. Let $\gamma:[0,\infty)\rightarrow(\R^2,g_f)$ be a timelike geodesic with $\gamma(0)=p$ and $\psi_\gamma(0)\neq 0$, then $\psi_\gamma(\tau)\neq0$ for any $\tau\geq0$. Otherwise, $f(x_0)\cdot\cosh\psi_\gamma(0)=f(x_{\gamma}(\tau))\cdot\cosh\psi_\gamma(\tau)=f(x_{\gamma}(\tau))\leq f(x_0)$, so we have that $\psi_\gamma(0)=0$, this is a contradiction. Moreover, timelike geodesic completeness of $(\mathbb{R}^{2},g_f)$ implies $-\infty<\psi_\gamma(\tau)<\infty$ for a geodesic $\gamma$, thus $t_{\gamma}$ and $x_{\gamma}$ are strictly monotone functions about $\tau\in\mathbb{R}$.

Suppose there exists another timelike geodesic $\zeta:[0,\infty)\rightarrow(\R^2,g_f)$ which emanates from $p$ and intersects $\gamma$ at a point other than $p$. Let $\gamma(\tau_{0})=\zeta(\tau_{1})$ be the first intersection point along $\zeta$. Without loss of generality, we may assume that $\tau_{0},\tau_{1}>0, 0<\psi_\gamma(0)<\psi_{\zeta}(0)<+\infty$. By \eqref{E2} we have
\begin{equation*}
\begin{split}
f(x_\gamma(\tau_{0}))\cosh\psi_\gamma(\tau_{0})=f(x_0)\cosh\psi_\gamma(0),\\
f(x_\zeta(\tau_{1}))\cosh\psi_{\zeta}(\tau_{1})=f(x_0)\cosh\psi_{\zeta}(0).
\end{split}
\end{equation*}
Since $x_\gamma(\tau_{0})=x_\zeta(\tau_{1})$, $\psi_\gamma(0)<\psi_{\zeta}(0)$, we have that $\psi_\gamma(\tau_{0})<\psi_{\zeta}(\tau_{1})$.
On the other hand, the geodesic $\gamma$ intersects the geodesic $\zeta$ at $\gamma(\tau_{0})=\zeta(\tau_{1})$ transversely such that $\{\dot{\gamma}(\tau_{0}),\dot{\zeta}(\tau_{1})\}$ gives the negative orientation of $\mathbb{R}^{2}$. Then $\psi_\gamma(\tau_{0})>\psi_{\zeta}(\tau_{1})$, this contradiction completes the proof.
\end{proof}

Now we could finish the proof of Theorem \ref{Thm1}:
\begin{proof}[Proof of Theorem \ref{Thm1}]
For $0<\varepsilon< 1$, we construct a smooth function $f:\mathbb{R}\rightarrow[2^{-1},2]$ such that
\begin{itemize}
  \item $f$ is 1-periodic w.r.t $x$, i.e., $f(x+1)=f(x)$,
  \item $[-\frac{\varepsilon}{4},\frac{\varepsilon}{4}]\subseteq f^{-1}(2^{-1})$ and $[\frac{\varepsilon}{2},1-\frac{\varepsilon}{2}]\subseteq f^{-1}(2)$.
\end{itemize}
Then, by Lemma \ref{Lemma1}, the associated metric $g_f$ satisfies
\begin{enumerate}
  \item $\mathbb{Z}^{2}$ acts on $(\mathbb{R}^{2},g_f)$, i.e., for any $(i,j)\in\mathbb{Z}^{2}$, $T_{(i,j)}:\mathbb{R}^{2}\rightarrow\mathbb{R}^{2};(t,x)\mapsto(t+i,x+j)$ are $g_f$-isometries;
  \item at each point of the plane, there are two smooth lightlike curves of $g_f$ and have the form of $\{(t,x(t))|t\in\mathbb{R}\}$; one of their slopes is larger than $\frac{1}{2}$, the other is less than $-\frac{1}{2}$;
  \item those points whose $x$-coordinate belongs to $[\frac{\varepsilon}{2}+j,1-\frac{\varepsilon}{2}+j],j\in\mathbb{Z}$ are timelike poles of $(\mathbb{R}^{2},g_f)$.
\end{enumerate}

Let $\pi:\mathbb{R}^{2}\rightarrow\mathbb{T}^{2}$ be the canonical cover. By (1) above, $(\T^2,g_f):=(\R^2,g_f)/\Z^2$ is a non-flat Lorentzian 2-torus, where the non-flatness follows from that $f$ is not constant; by (2) and the fact that the class A condition for $(\T^2,g_f)$ is equivalent to $m^+\neq m^-$ (see Proposition \ref{eq class A}), $(\mathbb{T}^{2},g_{f})$ is of class A; the last item implies that the volume of timelike poles is no less than $(1-\varepsilon)\text{Vol}(\mathbb{T}^{2},g_{f})$. Thus $(\mathbb{T}^{2},g_{f})$ is what we want.
\end{proof}

\begin{Rem}
Let $\mathbb{T}_{0}^{n-2}$ be the standard flat Riemannian torus, the Lorentzian product of $(\mathbb{T}^{2},g_{f})$ and $\mathbb{T}_{0}^{n-2}$ gives an $n$-dimensional Lorentzian torus $(\mathbb{T}^{n},g)$ satisfying all conditions except the dimension restriction. It is easy to see that the above Lorentzian torus $(\T^n,g)$ is actually a class A spacetime in the sense of \cite{Su1}.
\end{Rem}

\section{$P$-Motions on globally hyperbolic planes}
In this section, we discuss certain special isometries on a globally hyperbolic plane $(M,g)$ (i.e., $M$ is homeomorphic to $\mathbb{R}^2$) and some of their properties, our method goes back to H. Busemann. In \cite{Busemann1, Busemann2, Busemann3}, H. Busemann used isometries to study the distributions of conjugate points along geodesics in the so called $G$-spaces, which are metric spaces sharing some axioms with Finsler manifolds. As we shall see, H. Busemann's methods can also be very effective in the study of globally hyperbolic spacetimes.

We begin by giving some general definitions and properties. As is well known, isometries or motions on a metric space are distance preserving maps of the space onto itself. Similarly, we have
\begin{Def}\label{Def_$P$-motion}
Let $(M,g)$ be a globally hyperbolic spacetime, a motion on $(M,g)$ is a distance preserving map $\phi$ of $M$ onto itself, i.e., for any $p,q\in M, d(\phi p,\phi q)=d(p,q)$. Furthermore, we call $\phi$ a $P$-motion if it is a motion and $p\ll\phi p$ for all $p\in M$.
\end{Def}

\begin{Rem}
Here and after, we note that
\begin{itemize}
  \item contrary to the Riemannian case, distance preserving maps between spacetimes are in general not diffeomorphisms, however, under the global hyperbolicity condition, they actually are, see \cite[Theorem 4.17]{B-E-E}. So motions on globally hyperbolic spacetimes are Lorentzian isometries.

  \item for $p\in M, S\subseteq M$ or a causal curve $\gamma$, we use $\phi^{k}p,\phi^{k}S$ and $\phi^{k}\gamma,k\in\Z$ to denote the corresponding objects under $\phi^{k}$, i.e., $k$ times composite of $\phi$.
\end{itemize}
\end{Rem}

For a motion $\phi$ and a point $p\in M$, let $I^\pm(\phi,p):=\bigcup_{i\in \Z}I^\pm(\phi^i p)$. Since $I^\pm(p)$ are open for any $p\in M$ and $I^{\pm}(\phi p)=\phi I^{\pm}(p)$, thus $I^\pm(\phi,p)$ are $\phi$-invariant nonempty open sets. In addition, if $\phi$ is a $P$-motion, then $\{\phi^i p\}_{i\in\Z}\subseteq I^+(\phi,p)\cap I^-(\phi,p)$. These notions lead to

\begin{Pro}\label{prop4-4}
Let $\phi$ be a $P$-motion on a globally hyperbolic spacetime $(M,g)$ and $p$ a point for which
\begin{equation}\label{sup-1}
d(p,\phi p)=\sup_{q\in I^+(\phi,p)\cap I^-(\phi,p)}d(q,\phi q),
\end{equation}
then for any maximal timelike geodesic segment $T=T(p,\phi p)$ from $p$ to $\phi p$, the concatenated timelike curve $\gamma:=\bigcup_{i\in\Z}\phi^i T$ is a smooth $\phi$-invariant timelike geodesic. Moreover, $d(q,\phi q)=d(p,\phi p)$ for any $q\in\text{Im}(\gamma)$.
\end{Pro}

\begin{proof}
For any interior point $q\in T, p\ll q\ll \phi p$, thus $q\in I^+(\phi,p)\cap I^-(\phi,p)$. Since $\phi$ is a $P$-motion, then by inequality \eqref{rt}
\begin{equation*}
d(p,\phi p)=d(p,q)+d(q,\phi p)=d(q,\phi p)+d(\phi p, \phi q)\leq d(q,\phi q)\leq d(p,\phi p).
\end{equation*}
This implies
\begin{equation}\label{smooth}
d(q,\phi p)+d(\phi p,\phi q)= d(q,\phi q)= d(p,\phi p)
\end{equation}
and $T':=T(q,\phi p)\cup T(\phi p,\phi q)$ is a maximal timelike geodesic segment. Thus $\gamma$ is smooth at $\phi p$. By induction, $\gamma$ is a smooth timelike geodesic and is $\phi$-invariant by its definition. In general for any $q\in\text{Im}(\gamma)$, there exists $i\in\Z$ such that $\phi^{i}q\in T$, we could replace $q$ by $\phi^{i}q\in T$ and use equality \eqref{smooth} to complete the proof.
\end{proof}

\begin{Def}\label{Def_axis}
Let $\phi$ be a $P$-motion on a globally hyperbolic spacetime $(M,g)$. We call $\phi$ an axial motion on $(M,g)$ and $\gamma$ an axis of $\phi$ if $\gamma$ is a maximal timelike geodesic and $\phi$ maps $\gamma$ onto itself.
\end{Def}

\begin{Rem}\label{rem axis}
If $\phi$ is an axial motion on $M$ and $\gamma$ is an axis of $\phi$, then $\phi\gamma(t)=\gamma(t+a)$ for some $a>0$. For any $n\in\Z$ and $p\in\text{Im}(\gamma)$, $\phi^n p\in\text{Im}(\gamma)$ and $d(p,\phi^n p)=nd(p,\phi p)\rightarrow\infty$ as $n\rightarrow\infty$. Hence an axis is a complete timelike line. We also note that
\begin{itemize}
  \item if $\gamma$ is an axis of $\phi$, then $\gamma$ is also an axis of $\phi^l$ for $l\in\Z_+$.
  \item $\gamma$ constructed in Proposition \ref{prop4-4} is in general not maximal thus not an axis of $\phi$.
\end{itemize}
\end{Rem}

Notice that for a timelike line $\gamma$, $I^+(\gamma)\cap I^-(\gamma)$ is an open neighborhood of $\text{Im}(\gamma)$. Moreover, if $\gamma$ is an axis of a $P$-motion $\phi$, then $I^\pm(\gamma)=I^\pm(\phi,p)$ for any $p\in\text{Im}(\gamma)$ and we have

\begin{Pro}\label{prop4-1}
Let $\phi$ be an axial motion on a globally hyperbolic spacetime $(M,g)$. If $\gamma$ is an axis of $\phi$, then for any $l\in\Z_+$ and any $p\in\text{Im}(\gamma)$,
\begin{equation}\label{sup-2}
d(p,\phi^{l}p)=\sup_{q\in I^+(\gamma)\cap I^-(\gamma)}d(q,\phi^{l}q).
\end{equation}
\end{Pro}

\begin{proof}
Assume $l=1$, for any $p\in\text{Im}(\gamma), q\in I^+(\gamma)\cap I^-(\gamma)$, by the definitions of $I^\pm (\gamma)$ and Remark \ref{rem axis}, there exist two integers $m_1<m_2$ such that $q\in I^+(\phi^{m_1}p)\cap I^-(\phi^{m_2}p)$. Thus for any $i\in\Z_{+}, \phi^i q\in I^+(\phi^{m_1+i}p)\cap I^-(\phi^{m_2+i}p)$. By Proposition \ref{globally-hyperbolic},
\begin{equation} \label{E_dist}
d(\phi^{m_1}p,q)+d(q,\phi^{m_2}p)<\infty.
\end{equation}
Since $\phi$ is a $P$-motion, for any $k\in\Z_+$, we have
\begin{eqnarray*}
\phi^{m_1}p&\ll &q\ll \phi q\ll\cdots\ll \phi^k q\ll \phi^{m_2+k}p,  \\
\phi^{m_1}p&\ll &\phi^{m_1+1}p\ll\cdots\ll \phi^{m_2+k}p.
\end{eqnarray*}
Thus the reverse triangle inequality \eqref{rt} implies
\begin{eqnarray*}
(k+m_2-m_1)d(p,\phi p) &=& d(\phi^{m_1} p,\phi^{m_2+k} p) \\
&\geq& d(\phi^{m_1}p,q)+\sum\limits_{i=0}^{k-1} d(\phi^i q,\phi^{i+1} q)+d(\phi^k q,\phi^{k+m_2}p)  \\
&=& kd(q,\phi q)+d(\phi^{m_1}p,q)+d(q,\phi^{m_2}p).
\end{eqnarray*}
Dividing both sides by $k$, we have
\begin{equation*}
(1+\frac{m_2-m_1}{k})d(p,\phi p)\geq d(q,\phi q)+\frac{d(\phi^{m_1}p,q)+d(q,\phi^{m_2}p)}{k}.
\end{equation*}
Letting $k\rightarrow\infty$, together with inequality \eqref{E_dist}, leads to equation \eqref{sup-2} when $l=1$. For the general case, the conclusion follows from the fact that $\gamma$ is also an axis of $\phi^l(l\in\Z_+)$, see Remark \ref{rem axis}.
\end{proof}

Now we restrict ourself to consider globally hyperbolic planes $(M,g)$, that is, $M$ is homeomorphic to $\mathbb{R}^2$, instead of general globally hyperbolic spacetimes. In this case, contrary to Remark \ref{rem axis}, we could improve the conclusion of Proposition \ref{prop4-4} into

\begin{The}\label{prop4-5}
Let $\phi$ be an orientation preserving $P$-motion on a globally hyperbolic plane $(M,g)$ and $p$ a point satisfying equation \eqref{sup-1}, then there exists a unique maximal timelike geodesic segment $T=T(p,\phi p)$ from $p$ to $\phi p$ and $\gamma:=\bigcup_{i\in\Z}\phi^i T$ is a timelike line and hence an axis of $\phi$.
\end{The}

In the proof of Theorem \ref{prop4-5}, we shall use two elementary facts which are effective when study behaviors of maximal causal geodesics on Lorentzian planes.
\begin{Lem}\cite[Section 4]{Jin-Cui1}\label{Lemma_Morse}
Let $(M,g)$ be a globally hyperbolic spacetime, then
\begin{enumerate}
  \item Morse's crossing lemma: Two timelike maximal geodesic segments cannot intersect twice in $(M,g)$, except that the intersections occurred  at their end points.
  \item Curve lengthening lemma: Every timelike maximizer can be parametrized as a smooth geodesic. Thus, if a timelike curve $\gamma:[a,b]\rightarrow(M,g)$ contains a corner in its interior, it can not be maximal.
\end{enumerate}
\end{Lem}

\textit{Proof of Theorem \ref{prop4-5}:}
We begin to prove the uniqueness of $T(p,\phi p)$ if $p$ satisfies equation \eqref{sup-1}. Assume this is not the case, there are two distinct maximal timelike geodesic segments $T_1,T_2$ from $p$ to $\phi p$. By Morse's crossing lemma \ref{Lemma_Morse}, $T_1, T_2$ have no common points except $p$ and $\phi p$. Denote by $\eta$ the curve $T_1$ traversed from $p$ to $\phi p$ followed by $T_2$ traversed from $\phi p$ to $p$, then $\eta$ is an oriented simple closed curve. By Proposition \ref{prop4-4}, both $T_{1}\cup\phi T_1$ and $T_2\cup\phi T_2$ are smooth timelike geodesic segments which intersect transversally at $\phi p$, this implies that $\phi\eta$ is an oriented curve with the opposite orientation, in contradiction to the fact that $\phi$ preserves orientation (see Figure $1.$(A)).

Now to prove $\gamma:=\bigcup_{i\in\Z}\phi^{i}T$ is a timelike line, we only need to show that the arc $U^{k}:=\bigcup^{k-1}_{i=0}\phi^{i}T$ from $p$ to $\phi^{k}p$ is a maximal timelike geodesic segment for all $k\geq 1$. Obviously, this is true for $k=1$. By induction, assume that $U^{k-1}$ is maximal. If $U^k$ is not maximal, then global hyperbolicity implies that there is a maximal timelike geodesic segment $S$ from $p$ to $\phi^{k}p$ and
\begin{equation}\label{contra}
d(p,\phi^{k}p)=L^g(S)>k d(p,\phi p).
\end{equation}

Again by Lemma \ref{Lemma_Morse}, since both $U^{k-1}$ and $T(\phi^{k-1}p,\phi^{k}p)$ are maximal timelike geodesic segments, $S$ cannot intersect interior points of $U^{k-1}$ or $T(\phi^{k-1}p,\phi^{k}p)$. Moreover, $S$ cannot intersect $U^k$ at $\phi^{k-1}p$, because $T(p,\phi^{k-1}p)$ and $T(\phi^{k-1}p,\phi^{k}p)$ is unique as we have shown. Therefore, $S$ has no intersection points with $U^k$ except $p$ and $\phi^{k}p$.

We note that $S$ has no self-intersection. Otherwise we suppose $x\in S$ is such a point (see Figure $1.$(B)), then $d(x,x)>0$. On the other hand, $(M,g)$ contains no smooth closed timelike curves by global hyperbolicity, hence $d(x,x)=0$ for all $x\in M$, a contradiction!

Since $S$ has no self-intersection and $\phi$ preserves orientation, $S$ and $\phi S$ intersect at exactly one point $q$, thus $\phi q=\phi S\cap \phi^2S$, by the assumptions we get that $\phi S$ is a maximal timelike segment from $\phi p$ to $\phi^{k+1}p$ and $d(\phi p,\phi^kp)=d(p,\phi^{k-1}p)=(k-1)d(p,\phi p)$ (see Figure $1.$(C)).

By inequality \eqref{contra} and above discussions, we have
\begin{eqnarray*}
kd(p,\phi p) & < & d(p,\phi^kp)=L^g(S)=L^g(\phi S) \\
& = & d(\phi p,q)+d(q,\phi q)+d(\phi q,\phi^{k+1}p)  \\
& = & d(\phi p,q)+d(q,\phi^kp)+d(q,\phi q)  \\
& \leq &  d(\phi p,\phi^kp)+d(q,\phi q)  \\
& = & (k-1)d(p,\phi p)+d(q,\phi q).
\end{eqnarray*}
Thus $d(p,\phi p)<d(q,\phi q)$, which contradicts the fact that $p\ll q\ll p^k$ and the equation \eqref{sup-1}.

\begin{figure}
    \begin{subfigure}[b]{.3\linewidth}
        \centering
        \begin{tikzpicture}[scale=0.6,auto=right]
          \draw (-0.2,0.2)  to[out=-45,in=315] (0,0)node[scale=2]{.};
          \draw (-0.2,-0.14) to[out=40,in=210] (0,0);
          \draw (0,0) node[align=center, below=2pt]{\small$p$} to[out=30,in=135] node[above=1pt]{\small $T_2$}(3,0)node[scale=2]{.};
          \draw [->](2,-0.3) arc (-90:90:10pt);
          \draw [->](5,0.4) arc (90:-90:10pt);
          \draw (0,0) to[out=315,in=210] node[below=1pt]{\small $T_1$} (3,0);
          \draw (3,0) node[align=center, below=1pt]{\small$p^1$} to[out=30,in=135] node[above=1pt]{\small $\phi T_1$} (6,0)node[scale=2]{.};
          \draw (3,0) to[out=315,in=210] node[below=1pt]{\small $\phi T_2$} (6,0);
          \draw (6,0) node[align=center, below=1pt]{\small$p^2$} to[out=-45,in=135] (6.2,-0.2);
          \draw (6,0) to[out=30,in=220] (6.2,0.14);
        \end{tikzpicture}
        \caption{}
        \label{subfig:A}
    \end{subfigure}
    \begin{subfigure}[b]{0.3\linewidth}
        \centering
        \begin{tikzpicture}[scale=0.7,auto=right]
          \draw (0,0) node[scale=2]{.} node[below=2pt]{\small$x$} to[out=20,in=0] coordinate[pos=0.5](S) (0,3);
          \draw (0,3) to[out=180,in=160] (0,0);
          \draw (S) node[right]{\small$S$};
          \draw (-1,-0.2) to[out=8,in=200] (0,0);
          \draw (0,0) to[out=-20,in=172] (1,-0.2);
        \end{tikzpicture}
        \caption{}
        \label{subfig:B}
    \end{subfigure}
    \begin{subfigure}[b]{0.3\linewidth}
        \centering
        \begin{tikzpicture}[scale=0.6,auto=left]
          \draw (-0.5,0) --(8.5,0);
          \draw (4,0)node[below]{$\gamma$};
          \draw (0,0) node[below=2pt]{\small $p$} to coordinate[pos=1/8](p1) coordinate[pos=1/4](p2) coordinate[pos=3/4](p8) coordinate[pos=7/8](p9) coordinate[pos=1](p10) (8,0);
          \fill[black] (0,0) circle (1pt);
          \fill[black] (p1) circle (1pt);
          \fill[black] (p2) circle (1pt);
          \fill[black] (p8) circle (1pt);
          \fill[black] (p9) circle (1pt);
          \fill[black] (p10) circle (1pt);
          \draw (p1) node[below]{\small $p^1$};
          \draw (p2) node[below]{\small $p^2$};
          \draw (p8) node[below]{\small $p^{k}$};
          \draw (p9) node[below]{\small $p^{k+1}$};
          \draw (p10) node[below]{\small $p^{k+2}$};
          \draw (5.4,0.8) node[right]{\tiny $S$};
          \draw (6.4,0.8) node[right]{\tiny $\phi S$};
          \draw (7.4,0.8) node[right]{\tiny $\phi^2S$};
          \draw[name path=line 1] (0,0) to[out=50,in=110] (p8);
          \draw[name path=line 2] (p1) to[out=50,in=110] (p9);
          \draw[name path=line 3] (p2) to[out=50,in=110] (p10);
          \fill[black,name intersections={of=line 1 and line 2,total=\t}]
            \foreach \s in {1,...,\t}{(intersection-\s) circle (2pt) node[black,above]{\small$q$}};
          \fill[black,name intersections={of=line 2 and line 3,total=\t}]
            \foreach \s in {1,...,\t}{(intersection-\s) circle (2pt) node[black,above]{\small$q^1$}};
        \end{tikzpicture}
        \caption{}
        \label{subfig:C}
    \end{subfigure}
    \caption{}
    \label{Figure1}
\end{figure}

\qed

\begin{Rem}
By Theorem \ref{prop4-5}, on a globally hyperbolic plane, the concepts of orientation preserving $P$-motion and axial motion coincide.
\end{Rem}

\section{$P$-motions and class A 2 tori}
In this section, we use the results obtained in Sections 2 and 4 to further discuss $P$-motions on $(\R^2,g)$, where $(\R^2,g)$ is the Abelian cover of a class A 2-torus $(\T^2,g)$.

Let $\phi$ be a $P$-motion on $(\R^2,g)$ and $\gamma$ an axis of $\phi$,
\begin{itemize}
  \item by Lemma \ref{thm_Schelling}, $\gamma$ has an asymptotic direction $\alpha\in[m^-,m^+]$,
  \item if the asymptotic direction of $\gamma$ falls in $(m^-,m^+)$, then by Lemma \ref{lem_I_gamma}, for any $p\in\text{Im}(\gamma), I^\pm(\phi,p)=I^\pm(\gamma)=\R^2$.
\end{itemize}
Thus we obtain an immediate corollary of Proposition \ref{prop4-1}:

\begin{Lem} \label{cor4-1}
Let $(\R^2,g)$ be the Abelian cover of a class A Lorentzian 2-torus. If $\phi$ is a $P$-motion on $(\R^2,g)$ and $\gamma$ is an axis of $\phi$ with asymptotic direction $\alpha\in(m^-,m^+)$, then for any $l\in\Z_{+}$ and any $p\in\text{Im}(\gamma)$,
\begin{equation*}
d(p,\phi^{l}p)=\sup_{q\in\R^2}d(q,\phi^{l}q).
\end{equation*}
\end{Lem}

\begin{Lem} \label{lem_bf_levelset}
Let $(\R^2,g)$ be the Abelian cover of a class A 2-torus and $\gamma:[0,+\infty)\rightarrow \R^2$ a subray of a timelike line with an asymptotic direction in $(m^-,m^+)$. Then for any $p\ll q$ lying on $\text{Im}(\gamma)$,
\begin{equation*}
d(p,q)=d(K_{\infty}(p,\gamma),K_{\infty}(q,\gamma)).
\end{equation*}
\end{Lem}

\begin{proof}
By the assumptions and Lemma \ref{lem_I_gamma}, $I^-(\gamma)=\R^2$ and $b_{\gamma}$ is finite everywhere. By (3) of Proposition \ref{pro-buse}, for any $x\in K_{\infty}(p,\gamma)$ and $y\in K_{\infty}(q,\gamma)$,
\begin{itemize}
  \item if $y\in J^{+}(x)$, then $d(p,q)=b_{\gamma}(q)-b_{\gamma}(p)=b_{\gamma}(y)-b_{\gamma}(x)\geq d(x,y)$,
  \item if $y\notin J^+(x)$, then $d(x,y)=0$, $d(p,q)\geq d(x,y)$ obviously holds.
\end{itemize}
The lemma follows directly from equation \eqref{diss}.
\end{proof}

\begin{Lem}\label{lem_foot}
If $f$ is a foot of $p$ on a set $S$ in a globally hyperbolic spacetime and $q$ is a point satisfying $p\ll q\ll f$ and $d(p,f)=d(p,q)+d(q,f)$, then $f$ is the unique foot of $q$ on $S$.
\end{Lem}

\begin{proof}
By the assumptions, for any point $x\in S$,
\begin{itemize}
  \item if $x\notin J^+(q)$, then $d(q,x)=0<d(q,f)$,
  \item if $x\in J^+(q)$, then $d(q,x)\leq d(p,x)-d(p,q)\leq d(p,f)-d(p,q)=d(q,f)$.
\end{itemize}
Thus $f$ is a foot of $q$ and $d(q,x)=d(q,f)>0$ for some $x\in S$ only if $d(p,x)=d(p,q)+d(q,x)$. By global hyperbolicity and $p\ll q\ll f$, let $T(p,q)$ and $T(q,x)$ be two timelike maximal geodesic segments connecting $p,q$ and $q,x$ respectively. If $x\neq f$, then $T(p,q)\cup T(q,x)$ has a corner at $q$, by Lemma \ref{Lemma_Morse} so that can not be maximal, $d(p,x)>d(p,q)+d(q,x)=d(p,f)$. This is a contradiction since $d(p,f)=d(p,S)$, hence $f$ is unique.
\end{proof}

\begin{Def}\label{parallel}
Two timelike lines $\gamma$ and $\zeta$ in a globally hyperbolic spacetime are parallel if each future (resp. past) directed subray of $\gamma$ is a co-ray to a future (resp. past) directed subray of $\zeta$ and vice versa.
\end{Def}
For the definition of co-ray, we refer the readers to \cite{Galloway} and \cite{Bangert94}.

In the following, we shall prove that two axes of an axial $P$-motion are parallel under certain conditions.

\begin{Pro}\label{prop4-2}
Let $(\R^2,g)$ be the Abelian cover of a class A 2-torus and $\gamma$ an axis of a $P$-motion $\phi$ with asymptotic direction $\alpha\in(m^-,m^+)$. If $\zeta$ is also an axis of $\phi$, then $\zeta$ parallels to $\gamma$, hence has the same asymptotic direction $\alpha$.
\end{Pro}

\begin{proof}
We shall only consider future directed curves, the proof for past directed curves is completely similar. Let $S$ be a future directed subray of $\gamma$, for any $q\in Im(\zeta)$ and $k\in\Z$, we have
\begin{equation*}
K_{\infty}(\phi^{k}q,S)=K_{\infty}(\phi^{k}q,\phi^kS)=\phi^kK_{\infty}(q,S).
\end{equation*}
By the assumptions, $K_{\infty}(q,S)$ intersects $\gamma$ at a point $p$, then
\begin{equation*}
\phi^k p\in\phi^kK_{\infty}(q,S)=K_{\infty}(\phi^{k}q,S).
\end{equation*}
So $\phi^{k}p$ is the intersection point of $K_{\infty}(\phi^{k}q,S)$ and $S$. From Lemmas \ref{cor4-1} and \ref{lem_bf_levelset},
\begin{eqnarray*}
d(\phi^{-1}p,\phi^{k}p)&=& \sup_{y\in\R^2}d(y,\phi^{k+1}y)\\
                       &=& d(\phi^{-1}q,\phi^{k}q)\\
                       &\leq& d(\phi^{-1}q,K_{\infty}(\phi^{k}q,S))\\
                       &\leq& d(\phi^{-1}p,\phi^{k}p).
\end{eqnarray*}
Thus, $\phi^{k}q$ is a foot of $\phi^{-1}q$ on $K_{\infty}(\phi^{k}q,S)$. Together with Lemma \ref{lem_foot}, we know that $\phi^{k}q$ is the only foot of $q$ on $K_{\infty}(\phi^{k}q,S)$. This implies that the co-ray from $q$ to $\gamma$ is the future directed subray of $\zeta$ emanating from $q$. Therefore, we conclude that $\gamma$ and $\zeta$ are parallel and have the same asymptotic direction, otherwise they would intersect by Lemma \ref{lem_I_gamma}.
\end{proof}

\begin{Cor} \label{prop4-3}
Let $(\R^2,g)$ be the Abelian cover of a class A 2-torus and $\phi$ an orientation preserving $P$-motion on $(\R^2,g)$. If $\gamma$ is an axis of $\phi^k(k>1, k\in\Z_+)$ with asymptotic direction in $(m^-,m^+)$, then $\gamma$ is also an axis of $\phi$.
\end{Cor}
\begin{proof}
Suppose that $\gamma$ was not an axis of $\phi$, then $\gamma_1:=\phi\gamma\neq\gamma$ and
\begin{equation*}
  \phi^k\gamma_1=\phi^k\phi\gamma=\phi\phi^k\gamma=\phi\gamma=\gamma_1.
\end{equation*}
So $\gamma_1$ is an axis of $\phi^k$ and by Proposition \ref{prop4-2}, $\gamma_1$ is parallel to $\gamma$. Let $Q$ be the half plane bounded by $\gamma$ and containing $\gamma_1$. Since $\phi$ preserves the orientation, $\phi Q$ is the half plane bounded by $\gamma_1$ and not containing $\gamma$. By repeating this argument, we conclude that $\phi^k Q(k>1)$ is a half plane bounded by $\phi^k\gamma$ and not containing $\gamma, \phi\gamma, \phi^2\gamma,...,\phi^{k-1}\gamma$, which contradicts the fact that $\phi^k\gamma=\gamma$.
\end{proof}

We call $\Gamma=\{\psi_t\}_{t\in\R}$ a one-parameter group of $P$-motions generated by $\psi$ if each of $\psi_t(t>0)$ and $\psi_t^{-1}=\psi_{-t}(t<0)$ is a $P$-motion and $\psi_0=id, \psi_1=\psi, \psi_{t+s}=\psi_t\circ\psi_s$ for all $t,s\in\R$. Then from the above properties, we obtain the following byproduct.
\begin{Pro}\label{prop4-6}
Let $(\R^2,g)$ be the Abelian cover of a class A 2-torus and $\psi$ an orientation preserving $P$-motion on $(\R^2,g)$ with an axis $\gamma$ whose asymptotic direction $\alpha\in(m^-,m^+)$. If there exists a one-parameter group $\Gamma=\{\psi_t\}_{t\in\R}$ of $P$-motions generated by $\psi$, a positive number $t_0$ and a point $p$ such that
\begin{equation*}
d(p,\psi_{t_0}(p))=\sup_{q\in\R^2}d(q,\psi_{t_0}(q))>0
\end{equation*}
then the orbit $\{\psi_t(p):t\in\R\}$ is a timelike line.
\end{Pro}
\begin{proof}
First of all, since $\psi_0=id$ and $\psi_t$ preserves orientation for $t$ small, then all $\psi_t$ preserve orientation for $t\in\R$. Since $\gamma$ is an axis of $\psi$ with asymptotic direction $\alpha\in(m^-,m^+)$ and also an axis of $\psi^k$ for $k\in\Z_+$, from Corallory \ref{prop4-3} we know that $\gamma$ is an axis of $\psi^{\frac{k}{l}}$ for $l\in\Z_+$. Hence, all $\psi_t(t>0)$ has an axis with asymptotic direction $\alpha\in(m^-,m^+)$.

On the other hand, from Theorem \ref{prop4-5}, there exists an axis $\zeta$ of $\psi_{t_0}$, which means that all the points $\psi_{kt_0}(p)=\psi_{t_0}^k(p), k\in\Z$, lie on the timelike line $\zeta$. Then $\zeta$ has asymptotic direction $\alpha\in(m^-,m^+)$ from Proposition \ref{prop4-2}. Again by Corallory \ref{prop4-3}, $\zeta$ is also an axis of $\psi_{\frac{t_0}{l}}$ for any positive number $l\in\Z$. Thus all the points $\psi_{\frac{kt_0}{l}}(p)=\psi_{\frac{t_0}{l}}^k(p)$ lie on $\zeta$ for all $k\in\Z$. Hence $\psi_t(p)$ lies on $\zeta$ for any $t\in\R$.
\end{proof}

\section{Proof of the Main Results}
In this section, we shall prove Theorem \ref{Thm3} by studying the exponential maps locating at timelike poles.

Recall that, similar to Riemannian manifolds, the exponential map $\text{exp}_p:T_pM\rightarrow M$ for a Lorentzian manifold $(M,g)$ is defined as: for $v\in T_pM$, let $\gamma_v(t)$ be the unique geodesic in $M$ with $\gamma_v(0)=p$ and $\dot{\gamma}_v(0)=v$, then the exponential $\text{exp}_p(v):=\gamma_v(1)$ if $\gamma_v(1)$ is defined. The notion of cut point will be useful to us in the following, see also \cite[Section 9.1]{B-E-E}:

\begin{Def}\label{Def_cutpoint}
Define the function $s:T_{-1}M\rightarrow \R\cup\{\infty\}$ by
\begin{equation*}
s(v)=\sup\{t\geq 0:d(\gamma_v(0),\gamma_v(t))=t\},
\end{equation*}
where $T_{-1}M:=\{v\in TM:v\text{ is future directed and }g(v,v)=-1\}$ is the unit future observer bundle and $\gamma_v(t):=\text{exp}_p(tv)$. If $0<s(v)<\infty$ and $\gamma_v(s(v))$ exists, then the point $\gamma_v(s(v))$ is called the future cut point of $p=\gamma_v(0)$ along $\gamma_v$. The past cut points may be defined similarly.
\end{Def}

\begin{Rem}\label{s-prop}
By \cite[Proposition 9.7]{B-E-E}, the function $s:T_{-1}M\rightarrow \R\cup\{\infty\}$ is
\begin{itemize}
  \item positive, i.e., $s(v)>0$ for all $v\in T_{-1}M$, if $(M,g)$ is strongly causal,
  \item lower semicontinuous if $(M,g)$ is globally hyperbolic.
\end{itemize}
\end{Rem}

\begin{Pro}\label{prop_nocutpoint}
Let $(M,g)$ be a globally hyperbolic plane (i.e., $M$ is homeomorphic to $\mathbb{R}^2$) and $p\in M$ a timelike pole, then $p$ has no future (resp. past) causal conjugate or cut points.
\end{Pro}

\begin{proof}
First we note that any two dimensional Lorentzian manifold has no null conjugate points by \cite[Lemma 10.45]{B-E-E}. Moreover, $(M,g)$ has no null cut points: otherwise from \cite[Theorem 9.15]{B-E-E}, there would exist two null geodesic segments on $M$ intersecting only at $p$ and a cut point of $p$, contradiction!

By time duality, we shall only show that $p$ has no future timelike cut points. Assume that $q$ is a future timelike cut point of $p$, then using \cite[Theorem 9.12]{B-E-E}, there exist two distinct maximal timelike geodesic segments $\gamma_{v_i}:[0,s(v_i)]\rightarrow M$ connecting $p$ with $q$ such that $\gamma_{v_i}(0)=p, \gamma_{v_i}(s(v_i))=q, v_i\in T_{-1}M|_p, i=1,2$ and $s(v_1)=s(v_2)=d(p,q)$.

Since $p$ is a timelike pole, i.e., $q$ is not conjugate to $p$ along any timelike geodesic, then by \cite[Corollary 10.44]{B-E-E} there are only finitely many timelike geodesics from $p$ to $q$. Hence we may suppose that $\gamma_{v_1},\gamma_{v_2}$ are neighboring and denote by $D(\gamma_{v_1},\gamma_{v_2})$ the compact (topological) disk bounded by them, then $D(\gamma_{v_1},\gamma_{v_2})\backslash \{p,q\}\subseteq I^+(p)\cap I^-(q)$. Thus, for any $x\in D(\gamma_{v_1},\gamma_{v_2})\backslash \{p,q\}$, we have $0<d(p,x)<d(p,q)$.

On the other hand, we denote by $\widehat{v_1v_2}$ the compact set in $T_{-1}M|_p$ bounded by $v_1$ and $v_2$, so $0<s(v)\leq s(v_1)=s(v_2)=d(p,q)$ for all $v\in\widehat{v_1v_2}$. By Remark \ref{s-prop}, $s(v)$ is lower semicontinuous hence attains its minimum at $v_{1}'\in\widehat{v_1v_2}$, i.e., $s(v_1')=\min\limits_{v\in\widehat{v_1v_2}}s(v)$. If $t<s(v_1')$, then by the above construction, $\gamma_{v_1'}(t)\in D(\gamma_{v_1},\gamma_{v_2})$, thus $q'=\lim\limits_{t\rightarrow s(v_1')^-}\gamma_{v_1'}(t)$ exists and by Definition \ref{Def_cutpoint}, it is a future cut point of $p$.

If $s(v_1')=s(v_1)$, then $s(v)=s(v_1)$ for all $v\in\widehat{v_1v_2}$, this contradicts the fact that $\gamma_{v_1},\gamma_{v_2}$ are neighboring. Therefore we conclude that $s(v_1')<s(v_1)$. This implies that $q'$ lies in $D(\gamma_{v_1},\gamma_{v_2})\backslash \{p,q\}$, thus there exists $v_2'\neq v_1'$ on $\widehat{v_1v_2}$ such that $s(v_1')=s(v_2')$ and $q'=\gamma_{v_2'}(s(v_2'))$. Repeating the above argument with $q,v_{1},v_{2}$ replaced by $q',v_1',v_2'$, we get some $v_0\in\widehat{v_1v_2}$ such that $s(v_0)<s(v_1')=\min\limits_{v\in\widehat{v_1v_2}}s(v)$. This contradiction completes the proof.
\end{proof}

An immediate corollary of Proposition \ref{prop_nocutpoint} is the following
\begin{Pro}\label{cor_exponential_map}
Let $p$ be a timelike pole on a globally hyperbolic plane $(M,g)$ and
$$
\mathcal{F}^{+}(p)=\{tv|v\in T_{-1}M|_p, 0<t<s(v)\}\subseteq T_pM,
$$
then the exponential map $\text{exp}_p$ is a diffeomorphism from $\mathcal{F}^{+}(p)$ onto $I^+(p)$.
\end{Pro}

\begin{proof}
Since $p$ is a timelike pole, i.e., $p$ has no conjugate points along any timelike geodesic passing through $p$, the exponential map $\text{exp}_p$ is a local diffeomorphism. Proposition \ref{prop_nocutpoint} actually shows that for any $v\in T_{-1}M|_p$, either $s(v)=\infty$ or the maximal timelike geodesic $\gamma_{v}(t),t\in[0,s(v))$, by strongly causal property, escapes from any compact subset of the plane. This implies that $\text{exp}_p:\mathcal{F}^{+}(p)\rightarrow I^+(p)$ is injective. $\text{exp}_p:\mathcal{F}^{+}(p)\rightarrow I^+(p)$ is surjective hence bijective since for any point $q\in I^+(p)$, there exists, by the global hyperbolicity of $(M,g)$, a maximal timelike geodesic $\gamma_{v}$ from $p$ to $q$ and $d(p,q)<s(v)$.
\end{proof}

\begin{Rem} \label{Rem_maximalgeodesic}
The above corollary also has an obvious dual version for $I^{-}(p)$. From the proof, every future (resp. past) inextendible timelike geodesic emanating from a timelike pole on a globally hyperbolic plane $(M,g)$ is maximal and not future (resp. past) imprisoned in any compact set, see also \cite[p.64]{B-E-E}.
\end{Rem}

Note that $(\mathbb{R}^2,g)$ denotes the Abelian cover of a class A 2-torus, we have the following definition.
\begin{Def}
Let $\phi$ be an orientation preserving $P$-motion on $(\R^2,g)$, we call $\phi$ a $P$-translation if the displacement function of $\phi$, $d_{\phi}(p):=d(p,\phi p),p\in \R^2$ attains its maximum on $\R^2$.
\end{Def}

From Theorem \ref{prop4-5}, there is an axis (i.e., a timelike line) $\gamma$ of $\phi$ such that $\phi\gamma(t)=\gamma(t+a)$, for $t\in\R$, where $a=\sup\limits_{p\in I^+(\gamma)\cap I^-(\gamma)}d(p,\phi p)$. If there are two distinct axes $\gamma_1,\gamma_2$ of $\phi$ on $\R^2$, then $\gamma_1,\gamma_2$ have the same asymptotic direction $\alpha\in [m^-,m^+]$. Otherwise $\gamma_1$ and $\gamma_2$ must have an intersection point $p$ by Lemma \ref{lem_I_gamma}. Since both $\gamma_1$ and $\gamma_2$ are $\phi$-invariant, then all $\phi^np, n\in\Z$ are intersection points, which contradicts Lemma \ref{Lemma_Morse}. Hence we say a $P$-translation $\phi$ has asymptotic direction $\alpha$ if one of its axes has asymptotic direction $\alpha$. Recall from Section 1 that $\Gamma$ denotes the group of deck transformations associated to $\pi:(\mathbb{R}^2,g)\rightarrow (\mathbb{T}^2,g)$. Now we set
\begin{equation*}
\hspace{0.5cm}D^+=\{\phi:\phi\in\Gamma \text{ is a $P$-translation with asymptotic direction in $(m^-,m^+)$}\},
\end{equation*}
and $D^-=\{\phi:\phi^{-1} \in D^+\}$, then our main result Theorem \ref{Thm3} is a direct consequence of this last

\begin{The}\label{main_thm2}
Let $(\R^2,g)$ be the Abelian cover of a class A 2-torus and $p$ a timelike pole on it. Then for every $\phi\in D^+\cup D^-$, there is an axis of $\phi$ passing through $p$. Moreover, all displacement functions $d_{\phi}(q)$ for $\phi\in D^+$ (resp. $d_{\phi}(q):=d(\phi q,q)$ for $\phi\in D^-$) attain their maximums at $p$.
\end{The}

\begin{proof}
The proof is an adaption of Busemann \cite[p.212]{Busemann1}, see also \cite[Proposition 4.1]{Innami}. We shall only consider the case $\phi\in D^+$ since the other case is proved by reversing the time orientation.

Let $\gamma$ be an axis of $\phi$, then $\gamma$ is a complete timelike line whose asymptotic direction in $(m^-,m^+)$. There is nothing to prove if $\gamma$ passes through $p$. Recall the definitions of $p<\gamma$ and $p>\gamma$ from Definition \ref{def_point_line}. Without loss of generality, we may assume that $p>\gamma$. Since $(\R^2,g)$ is the Abelian cover of a class A 2-torus $(\T^2,g)$, there exists a $\psi\in D^+$ such that $p<\psi\gamma$. So $p$ is contained in the strip bounded by $\gamma$ and $\psi\gamma$. Note that $\psi\gamma$ is also an axis of $\phi$ since
\begin{equation*}
\phi(\psi\gamma)(t)=\psi(\phi\gamma)(t)=\psi\gamma(t+a).
\end{equation*}
Thus, all $\phi^np,n\in\Z$ are contained in this strip since $\phi$ preserves the orientation of $\R^2$.

For $n\in\N (n\geq2)$ and each $k=0,1,...,n-1$, since $\phi^kp(\phi^0:=id)$ are timelike poles, we denote by $T(\phi^kp,\phi^{k+1}p)$ the unique maximal timelike geodesic segment from $\phi^kp$ to $\phi^{k+1}p$ and construct the piecewise smooth timelike curve $A_n=\bigcup\limits^{n-1}_{k=0}\phi^{k}T(p,\phi p)$; denote by $R_k:=R(\phi^kp,\phi^{k+1}p)$ the unique timelike ray starting from $\phi^kp$ and passing through $\phi^{k+1}p$, the existence of $R_k$ is guaranteed by Proposition \ref{cor_exponential_map}.

Since $\phi$ preserves orientation, $\phi^np$ stays in the same side of $R(\phi^{n-2}p,\phi^{n-1}p)$ as that of $R(p,\phi p)$ in which $\phi^2 p$ stays, for all $n\geq 2$. Hence by the argument used in the last part of the proof of Theorem \ref{prop4-5}, $R_k,k=0,...,n-1$ can not intersect $A_{n}$ transversally.
Also, we can choose one of $\gamma$ and $\psi\gamma$, denoted by $\zeta$, such that $\zeta$ and $R_k$ lie in the same side of $A_n$ for $k=0,1,...,n-1$ and $n\geq 2$.

Since $\gamma$ and $\psi\gamma$ have the same timelike asymptotic direction $\alpha$, by Lemma \ref{lem_I_gamma}, there exists $n_0\in\Z_+$ such that $I^+(p)\cap I^-(\phi^{n_0}p)$ intersects both sets $Im(\gamma)$ and $Im(\psi\gamma)$. Fixing this $n_{0}$, for any $z\in I^+(p)\cap I^-(\phi^{n_0}p)\cap \zeta$, we have
\begin{equation} \label{E_d}
d(p,z)+d(z,\phi^{n_0} p)<\infty.
\end{equation}

In the following, let $n>n_{0}$ and $D_n$ be the compact region bounded by $A_n$ and $B_n:=T(p,z)\cup T(z,\phi^{n-n_0}z)\cup T(\phi^{n-n_0}z,\phi^n p)$. From Remark \ref{Rem_maximalgeodesic} we know that $R_k$ must intersect $B_n$ exactly once for $k=0,1,\ldots,n-1$. Denote by $e_k$ the intersection point of $R_{k-1}$ and $B_n$ for $k=1,...,n-1$ and $e_0:=p$, $e_n:=\phi^np$ (see Figure 2).

\begin{figure}
    \begin{tikzpicture}
      \draw (-3,0) node[left]{$\zeta$} to (10.5,0);
      \draw (-1.5,5.5) node[right]{$I^+(p)$};
      \draw (-3,5.5) node[above]{$I^-(p)$};
      \draw (-2,0) node[below]{past};
      \draw (10,0) node[below]{future};
      \draw[fill={rgb:black,1;white,5}] (-2,5.5)--++ (45:1)--++(255:1.73205)--++(105:2)--++(255:1.73205)--++(45:1);
      \draw[fill={rgb:black,1;white,5}] (10,4.5)--++ (70:1)--++(-80:1.73205)--++(130:2)--++(-80:1.73205)--++(70:1);

      \draw[thick] (-2,5.5) to[out=-60,in=165] coordinate[pos=0](p) coordinate[pos=1/4](p1) coordinate[pos=1/2](p2) coordinate[pos=3/4](p3) coordinate[pos=1](p4)(3,2);
      \draw[thick] (8,3) to[out=30,in=225] coordinate[pos=0](p8) coordinate[pos=0.5](p9) coordinate[pos=1](p10) (10,4.5);
      \draw[thick,dashed] (p4) to[out=-15,in=210] (p8);
      \fill[blue] (p) circle (1pt);\fill[blue] (p1) circle (1pt);\fill[blue] (p2) circle (1pt);\fill[blue] (p3) circle (1pt);
      \fill[blue] (p8) circle (1pt);\fill[blue] (p9) circle (1pt);\fill[blue] (p10) circle (1pt);
      \draw (p) node[below=4pt]{$p$};\draw (p1) node[above=2pt]{\small $\phi p$};\draw (p2) node[above=2pt]{\small $\phi^2p$};\draw (p3) node[above=1.5pt]{\small $\phi^3p$};
      \draw (p8) node[left=3pt]{\small $\phi^{n-2}p$};\draw (p9) node[left]{\small $\phi^{n-1}p$};\draw (p10) node[above=3pt]{\small $\phi^{n}p$};
      \draw[red,thick] (p) to(1,0)node[black,below=2pt]{$z$};
      \draw[red,thick] (1,0) to (8,0);
      \draw[red,thick] (p10)to coordinate[pos=0.15](en-1) (8,0)node[black,below]{\small$\phi^{n-n_0}z$};
      \draw[dashed] (p1) to(2,0)node[above=2pt]{$e_1$};
      \draw[dashed] (p2) to(3.5,0)node[above=2pt]{$e_2$};
      \draw[dashed] (p3) to(5,0)node[above=2pt]{$e_3$};
      \draw[dashed] (p9) to(en-1) node[right]{\small$e_{n-1}$};
      \fill (1,0) circle(1pt);\fill (2.5,0) circle(1pt);\fill(4,0) circle(1pt);\fill (8,0)circle(1pt);\fill (en-1)circle(1pt);
      \draw (2.5,0) node[below]{$\phi z$};\draw (4,0) node[below]{$\phi^2 z$};
      \draw (5,2) node[above]{\large $A_n$};
      \draw (0,2) node[left=2pt]{\large $B_n$};

      \draw[|<->|] (2,-0.8) --(3.5,-0.8) node[pos=0.5,below]{\small$\delta_1$};
      \draw[<->|] (3.5,-0.8)--(5,-0.8) node[pos=0.5,below]{\small$\delta_2$};
    \end{tikzpicture}
    \caption{}
    \label{Figure:2}
\end{figure}

Since the chronological relation $\ll$ is transitive and $\phi$ is a $P$-translation, we have the following chronological relations:
\begin{eqnarray*}
p\ll \phi p\ll\cdots\ll \phi^{n-1} p\ll \phi^np,  \\
p\ll z\ll\phi z\ll\cdots\ll \phi^{n-n_0}z\ll \phi^np.
\end{eqnarray*}

Let $\delta_k$ be the $g$-arc length of the subcurve of $B_n$ joining $e_k$ and $e_{k+1}$ for each $k=0,1,2,\ldots,n-1$. Then we have
\begin{eqnarray*}
d(p,\phi p)+d(\phi p,e_1) & > & \delta_0,  \\
d(\phi p,\phi^2 p)+d(\phi^2 p,e_2) & > & d(\phi p,e_1)+\delta_1,  \\
d(\phi^2 p,\phi^3 p)+d(\phi^3 p,e_3) & > & d(\phi^2 p,e_2)+\delta_2,  \\
& \vdots & \\
d(\phi^{n-1} p,\phi^n p) & > & d(\phi^{n-1} p,e_{n-1})+\delta_{n-1}.
\end{eqnarray*}
Thus, $d(p,\phi p)+d(\phi p,\phi^2 p)+\cdots+d(\phi^{n-1} p,\phi^n p) > \delta_0+\delta_1+\cdots+\delta_{n-1}$. So we have
\begin{eqnarray*}
L(A_n)=n d(p,\phi p)  &>& L(B_n)=\sum\limits_{i=0}^{n-1}\delta_i   \\
&=& d(p,z)+\sum\limits_{i=0}^{n-n_0-1}d(\phi^iz,\phi^{i+1}z)+d(\phi^{n-n_0}z,\phi^np).
\end{eqnarray*}
Dividing both sides by $n$,
\begin{equation*}
d(p,\phi p)>(1-\frac{n_0}{n})d(z,\phi z)+\frac{d(p,z)+d(\phi^{n-n_0}z,\phi^np)}{n}.
\end{equation*}
Letting $n\rightarrow\infty$, together with inequality \eqref{E_d}, we get
\begin{equation*}
d(p,\phi p)\geq d(z,\phi z)=\sup_{y\in I^+(\phi,p)\cap I^-(\phi,p)}d(y,\phi y)
\end{equation*}
since $z\in\gamma$. Thus, by Proposition \ref{prop4-1} and Theorem \ref{prop4-5}, $d(p,\phi p)=d(z,\phi z)$ and there exists an axis of $\phi$ that passes through $p$.
\end{proof}

\end{document}